\documentclass[10pt]{amsart}
\usepackage{amsfonts,amssymb}
\usepackage[all]{xy}

\newtheorem{thm}{Theorem}[section]

\newtheorem{prop}[thm]{Proposition}
\newtheorem{propdef}[thm]{Proposition-Definition}

\theoremstyle{remark}

\newcommand*\isom{%
  \xrightarrow{\sim}%
}

\def\qq{\mathbb{Q}}

\def\rr{\mathbb{R}}
\def\zz{\mathbb{Z}}

\def\CC{\mathbb{C}}

\def\mm{\mathcal{M}}

\def\aa{\mathcal{A}}
\def\bb{\hat{\mathcal{B}}}

\def\QQ{\mathcal{Q}}

\def\jj{\mathcal{J}}
\def\oo{\mathcal{O}}
\def\vv{\mathcal{V}}
\def\hh{\mathcal{H}}
\def\xx{\mathcal{X}}
\def\yy{\mathcal{Y}}

\def\Pic{\mathcal{P}\mathrm{ic}}
\def\CH{\mathrm{CH}}

\numberwithin{equation}{section}

\begin{document}

\title{Normal functions and the height of Gross-Schoen cycles}
\author{Robin de Jong}



\begin{abstract}  We prove a variant of a formula due to S. Zhang relating the Beilinson-Bloch height of the Gross-Schoen cycle on a pointed curve with the self-intersection of its relative dualizing sheaf. In our approach the height of the Gross-Schoen cycle occurs as the degree of a suitable Bloch line bundle. We show that the Chern form of this line bundle is non-negative, and we calculate its class in the Picard group of the moduli space of pointed stable curves of compact type. The basic tools are normal functions and biextensions associated to the cohomology of the universal jacobian. 
\end{abstract}

\maketitle
\thispagestyle{empty}

\section{Introduction}

Let $X$ be a smooth geometrically connected curve of genus $g \geq 2$ over a field~$k$. Let $e=\sum_i a_i p_i$ be a divisor of degree one on $X$. Consider the following cycles of codimension two on the triple product $X^3$:
\[ \Delta_{123} = \{ (x,x,x) \colon x \in X \} \, , \]
\[ \Delta_{12} = \sum_i a_i \{ (x,x,p_i) \colon x \in X \} \, , \]
\[ \Delta_{13} = \sum_i a_i \{ (x,p_i,x) \colon x \in X \} \, ,  \]
\[ \Delta_{23} = \sum_i a_i \{ (p_i,x,x) \colon x \in X \} \, , \]
\[ \Delta_1 = \sum_{i,j} a_ia_j \{ (x,p_i,p_j) \colon x \in X \} \, , \]
\[ \Delta_2 = \sum_{i,j} a_ia_j \{ (p_i,x,p_j) \colon x \in X \} \, , \]
\[ \Delta_3 = \sum_{i,j} a_ia_j \{ (p_i,p_j,x) \colon x \in X \} \, , \]
and put
\[ \Delta_e = \Delta_{123} - \Delta_{12}-\Delta_{13}-\Delta_{23}+\Delta_1+\Delta_2+\Delta_3 \, . \]
Then $\Delta_e$ is a cohomologically trivial cycle on $X^3$, studied in detail by B. Gross and C. Schoen \cite{gs}. Assume that $k$ is either a number field or a function field of a curve, and assume that $X$ has semistable reduction over $k$. Then Gross and Schoen construct in \cite{gs} a canonical $\rr$-valued height $(\Delta_e,\Delta_e)$ associated to $\Delta_e$, fitting in a general approach due to A. Beilinson \cite{beil} and S. Bloch \cite{bljpaa}. An important feature of the height $(\Delta_e,\Delta_e)$ of the Gross-Schoen cycle is that its non-vanishing implies the non-triviality of the class of $\Delta_e$ in the Chow group $\mathrm{CH}^2(X^3)$ of codimension-$2$ cycles on $X^3$, and hence (granting the truth of the Beilinson-Bloch conjecture) the vanishing of the special value of a suitable $L$-series attached to $X^3$. The height of the Gross-Schoen cycle is conjectured to be non-negative \cite{zh}, Conjecture 1.4.1.

In a recent paper \cite{zh} S. Zhang derived an explicit formula for $(\Delta_e,\Delta_e)$, and used this formula to prove some important new results in this direction. We describe the formula, and simply refer to \cite{zh} for further discussion of its ramifications.

First of all, for each place $v$ of $k$ Zhang introduces a local $\rr$-valued invariant $\varphi(X_v)$ of $X$ at $v$, as follows. Let $\mu_v$ be the canonical Arakelov $(1,1)$-form \cite{ar} \cite{fa} on $X_v = X \otimes \bar{k}_v$ if $v$ is archimedean, and let $\mu_v$ be the canonical admissible metric from \cite{zhadm}, Section 3 on the reduction graph $R_v$ of $X$ at $v$ if $v$ is non-archimedean. If $v$ is an archimedean place,  
let $\Delta_{\mathrm{Ar}}$ be the Laplacian on $L^2(X_v,\mu_v)$ given by putting $\partial \overline{\partial} f = \pi i \Delta_\mathrm{Ar}(f) \mu_v$ for $C^\infty$ functions $f$. Let $(\phi_\ell)_{\ell=0}^\infty$ be an orthonormal basis of real eigenfunctions of $\Delta_\mathrm{Ar}$, with corresponding eigenvalues $0=\lambda_0<\lambda_1\leq \lambda_2 \leq \ldots$. He then puts
\begin{equation} \label{defphi}
\varphi(X_v) = \sum_{\ell>0} \frac{2}{\lambda_\ell} \sum_{m,n=1}^g \left| \int_{X_v} \phi_\ell \, \omega_m \wedge \bar{\omega}_n \right|^2 \, , 
\end{equation}
where $(\omega_m)_{m=1}^g$ is an orthonormal basis for the inner product on $\mathrm{H}^0(X_v, \Omega^1_{X_v})$ given by putting $( \alpha,\beta )= \frac{i}{2} \int_{X_v} \alpha \wedge \overline{\beta}$. If $v$ is a non-archimedean place, he puts
\[ \varphi(X_v) = -\frac{1}{4} \delta(X_v) + \frac{1}{4} \int_{R_v} g_{\mu,v}(x,x) ((10g+2)\mu_v - \delta_{K_v}) \, , \]
where $\delta(X_v)$ is the number of singular points in the geometric special fiber of $X$ at $v$, $g_{\mu,v}$ is the Green's function \cite{zhadm}, Section 3 for the canonical metric $\mu_v$ on $R_v$, and $K_v$ is the canonical divisor on $R_v$.

Let $x_e$ be the class of the degree zero divisor $(2g-2)e-K$ in the jacobian of $X$, where $K$ is a canonical divisor on $X$, and denote by $\hat{h}(x_e)$ the N\'eron-Tate height of $x_e$ in $\rr$ (we warn the reader that our notation at this point differs slightly from the notation in \cite{zh}). Let $(\omega,\omega)_a$ be the admissible self-intersection of the relative dualizing sheaf \cite{zhadm}, Section 4 of $X$ over $k$. Define local factors $Nv$  as follows: if $v$ is a real embedding, then $Nv=\mathrm{e}$, if $v$ is a pair of complex embeddings, then $Nv=\mathrm{e}^2$, and if $v$ is non-archimedean, then $Nv$ is the cardinality of the residue field at $v$, provided this residue field is finite, and $Nv=\mathrm{e}$ else. Zhang's formula is then the following.
\begin{thm} (S. Zhang, Theorem 1.3.1 of \cite{zh}) The equality
\[ (\Delta_e,\Delta_e) = \frac{2g+1}{2g-2} (\omega,\omega)_a + \frac{3}{2g-2} \hat{h}(x_e) - \sum_v \varphi(X_v) \log Nv
\]
holds in $\rr$. 
\end{thm}
The purpose of this paper is to prove a variant of Zhang's formula. The main aspect of our approach is that the archimedean contributions $\varphi(X_v)$ from Zhang's formula appear as a ``norm at infinity'' of a certain canonical isomorphism of line bundles related to $\Delta_e$, $\omega$ and $x_e$. Let $S$ be a quasiprojective variety over the complex numbers. Let $\pi \colon \xx \to S$ be a smooth projective family of curves of genus $g \geq 2$ over $S$, and assume that a flat divisor $e$ of relative degree one on $\xx/S$ is given. Let $\omega$ be the relative dualizing sheaf of $\xx/S$, and let $x_e$ be the   divisor class of relative degree zero on $\xx/S$ given by $(2g-2)e-c_1(\omega)$. Let $\langle \omega,\omega \rangle$ on $S$ be the Deligne pairing \cite{de} of $\omega$ with itself, and let $\langle x_e,x_e \rangle$ on $S$ be the Deligne pairing of $x_e$ with itself.
Let $\Delta_e$ be the flat relative Gross-Schoen cycle with base point $e$ on the triple self-fiber product of $\xx/S$. 

According to Bloch \cite{bl}, see also Section \ref{bloch} below, we have a canonical algebraic line bundle $\langle \Delta_e,\Delta_e \rangle$ associated to $\Delta_e$ on $S$. The formation of Bloch's line bundle is compatible with base change, and if $S$ is a smooth projective connected curve, then $\deg_S \langle \Delta_e,\Delta_e \rangle$ equals the height $(\Delta_e,\Delta_e)$ of the Gross-Schoen cycle on the generic fiber of $\xx/S$. 

Each of the algebraic line bundles $\langle \Delta_e,\Delta_e \rangle$, $\langle \omega,\omega \rangle$ and $\langle x_e,x_e \rangle$ is equipped with a canonical hermitian metric. For both Deligne pairings this is in \cite{de}, and for the Bloch pairing $\langle \Delta_e, \Delta_e \rangle$ we refer to Sections \ref{bloch} and \ref{connection_biext} below. Note that the archimedean $\varphi$-invariant defines a continuous real-valued function on $S$.
\begin{thm}  \label{thmA} There exists an isomorphism of algebraic line bundles
\[ \langle \Delta_e, \Delta_e \rangle^{\otimes 2g-2} \isom \langle \omega, \omega \rangle^{\otimes 2g+1} \otimes \langle x_e,x_e \rangle^{\otimes -3} \]
on $S$, canonical up to a sign, and having norm  $\exp(-(2g-2)\varphi)$.
\end{thm}
To prove Theorem \ref{thmA} we reduce to the universal case where $S$ is the complex moduli space $\Pic^1_g$ of pairs $(X,e)$ consisting of a complex smooth projective connected curve $X$ of genus $g$ together with a divisor class $e$ of degree one. We conjecture that an analogue of Theorem \ref{thmA} holds for quasiprojective varieties over complete discretely valued fields, featuring the non-archimedean $\varphi$-invariant introduced above. In combination with Theorem \ref{thmA}, such an analogue would yield Zhang's theorem as an immediate corollary.

We prove Theorem \ref{thmA} by considering normal functions and biextensions on a fibration of intermediate jacobians over $\Pic^1_g$. Our approach is rather different from the one pursued in Zhang's original paper \cite{zh}. We think that as a result of our approach the Hodge theoretic and motivic aspects of Zhang's paper become more visible. For example, the motive $M$ appearing in Section 5 of \cite{zh} corresponds to the intermediate jacobian $\jj(\wedge^3 H/H)$ studied in Section \ref{variation} below.

The hermitian line bundle $\langle \Delta_e,\Delta_e \rangle$ on $\Pic^1_g$ has a unique extension as a hermitian line bundle over $\Pic^{1,c}_g$, the universal Picard variety of degree one over the moduli space $\mm_g^c$ of stable curves of compact type. We denote this extension over $\Pic^{1,c}_g$ by the same symbol.
\begin{thm} \label{thmB} The line bundle $\langle \Delta_e,\Delta_e \rangle$ has non-negative Chern form and hence is nef on $\Pic^{1,c}_g$.
\end{thm}
As a corollary of Theorem \ref{thmB} we obtain the non-negativity of the height of the Gross-Schoen cycle for curves over function fields in characteristic zero with everywhere stable reduction of compact type. We note that this also follows from \cite{zh2} by an application of the Hodge Index Theorem.

To finish, we compute the first Chern class of the Bloch line bundle $\langle \Delta_e,\Delta_e \rangle$ restricted to the universal $1$-pointed stable curve of compact type $\mm^c_{g,1}$. Let $\psi$ be the first Chern class of the pullback, along the canonical section, of the relative dualizing sheaf of the universal curve over $\mm^c_{g,1}$, and let $\lambda$ be the first Chern class of the Hodge bundle. For $1\leq i \leq g-1$ an integer let $\Delta^{ \{x\} }_i$ be the boundary divisor in $\mm_{g,1}^c$ of which the generic point is a reducible curve with precisely one node and two smooth irreducible components, one of genus $i$ and the other of genus $g-i$, such that the marked point lies on the genus $i$ component minus the node. Let $\delta^{ \{x\} }_i$ be the class of $\Delta^{ \{x\} }_i$ in the Picard group of $\mm_{g,1}^c$. It is known \cite{ac} that $\mathrm{Pic}(\mm_{g,1}^c)$ is generated by $\psi$, $\lambda$ and the $\delta^{ \{x\} }_i$ (and freely so if $g \geq 3$). 
\begin{thm} \label{thmC} We have the equality
\[ \langle \Delta_e,\Delta_e \rangle = 6g \, \psi + 12 \, \lambda - \sum_{i=1}^{g-1} 6i\, \delta_{g-i}^{ \{x\} }  \]
in $\mathrm{Pic}(\mm_{g,1}^c)$.
\end{thm}
Observe that the class of $\langle \Delta_e,\Delta_e \rangle$ in the Picard group of $\mm_{g,1}^c$ is divisible by $6$. It would be interesting to have a description of a natural sixth root of $\langle \Delta_e,\Delta_e \rangle$.  

The set-up of the paper is as follows. We discuss some generalities on Bloch's pairing, biextensions and variations of polarized Hodge structure in Sections \ref{bloch}, \ref{connection_biext} and \ref{variation}, respectively. In Sections \ref{normal} and \ref{extension} we study normal functions associated to the well known Ceresa cycle in the jacobian of a curve. The connection with the Gross-Schoen cycle is made in Section \ref{proofthmA}, which also contains a proof of Theorem \ref{thmA}. Theorems \ref{thmB} and \ref{thmC} are proved in Section \ref{furtherextension}.

\section{Bloch's pairing between cycle classes} \label{bloch}

We begin by reviewing the construction, due to S. Bloch \cite{bl}, of a canonical algebraic line bundle $\langle z,w \rangle$ associated to a pair of relatively homologically trivial algebraic cycle classes $z,w$ on a family of smooth complex projective varieties. In Section \ref{connection_biext} we discuss an alternative formulation of essentially the same pairing in terms of biextensions and normal functions on the associated family of Griffiths intermediate jacobians. For a more extensive discussion of the topic, proofs and broader generality we refer to the sources \cite{mey}, \cite{ms}, \cite{sei} and of course \cite{bl}.

Let $S$ be a quasiprojective variety over $\CC$. Consider a smooth projective morphism $\pi \colon \xx \to S$ of relative dimension $d$ over $S$. For any positive integer $p \leq d$ let $\CH_o^p(\xx/S)$ denote the abelian sheaf (for the \'etale topology on $S$) of codimension~$p$ cycles on $\xx$, modulo rational equivalence, restricting to a homologically trivial cycle in each of the fibers of $\xx/S$. Let $\textbf{Pic}(S)$ be the groupoid of algebraic line bundles on $S$, and let $p,q$ be positive integers such that $p+q=d+1$.

In \cite{bl} a canonical pairing $\langle \cdot,\cdot \rangle \colon \CH^p_o(\xx/S) \times \CH^q_o(\xx/S) \to \textbf{Pic}(S)$ is constructed. Omitting most of the details (for which we refer to the above mentioned references), the construction of this pairing can be summarized as follows. Let $z,w$ be classes in $\CH^p_o(\xx/S)$ and $\CH^q_o(\xx/S)$, respectively. Then for any non-empty Zariski open subset $U$ of $S$, locally over $U$ the line bundle $\langle z, w\rangle$ is generated by symbols $\langle Z,W \rangle$, where $Z,W$ are representatives of $z,w$ over $U$, with disjoint support, such that the following relations hold.

Let $\langle Z,W \rangle $ and $\langle Z,W' \rangle$ be local symbols over $U$ and suppose that $W,W'$ are rationally equivalent over $U$. This essentially means that there exists a cycle $Y$ on $\xx_U$ and a rational function $f$ on $Y$ such that $W-W'=\mathrm{div}_Y(f)$; we assume we can choose $Y$ such that $Y$ and $Z$ intersect properly. Consider then the intersection cycle $Y. Z$; we assume that the support of $Y.Z$ is finite flat over $S$. Note that as $Z,W$ and $Z,W'$ have disjoint support the rational function $f$ restricts to an invertible regular function on $Y. Z$. Let $\sigma_Z(W-W')$ in $\oo_S^*(U)$ be the result of applying the norm map from $Y. Z$ to $S$ to the restriction of $f$ to $Y.Z$; then we put $\langle Z,W \rangle = \sigma_Z(W-W') \cdot \langle Z,W' \rangle$.

Likewise, if $\langle Z,W \rangle $ and $\langle Z',W \rangle$ are local symbols over $U$ with $Z,Z'$ rationally equivalent over $U$, with a cycle $Y'$ on $\xx_U$ intersecting $W$ properly and a rational function $g$ on $Y'$ such that $Z-Z'=\mathrm{div}_{Y'}(g)$, we put $\langle Z,W \rangle = \sigma_W(Z-Z') \cdot \langle Z',W \rangle$, where $\sigma_W(Z-Z')$ in $ \oo_S^*(U)$ is the result of applying the norm map from $Y'.W$ to $S$ to the invertible regular function obtained by restricting $g$ to $Y'.W$.

We mention the cocycle relations
\[ \sigma_Z(W-W') \cdot \sigma_{W'}(Z-Z') = \sigma_W(Z-Z') \cdot \sigma_{Z'}(W-W') \]
in $\oo_S^*(U)$ for $Z,W,Z',W'$ as above.

In \cite{bl}, a `torseur' $\mathbb{E}$ is constructed over the product $\CH^p_o(\xx/S) \times \CH^q_o(\xx/S)$. The above $\langle z, w\rangle$ are the fibers of this torseur. The various properties of the torseur $\mathbb{E}$ discussed in \cite{bl} readily yield the following proposition.
\begin{prop} \label{basic}
(a) (Symmetry) Let $z,w$ be in $\CH^p_o(\xx/S)$ and $\CH^q_o(\xx/S)$, respectively. Then there exists a canonical isomorphism $\langle z,w \rangle \isom \langle w,z \rangle$ of line bundles on $S$. \\
(b) (Bi-additivity) Let $z,z'$ be elements of $\CH^p_o(\xx/S)$ and let $w,w'$ be elements of $\CH^q_o(\xx/S)$. Then there are canonical isomorphisms
\[ \langle z+z', w \rangle \isom \langle z,w \rangle \otimes \langle z', w \rangle \, , \qquad
\langle z, w +w'\rangle \isom \langle z,w \rangle \otimes \langle z, w' \rangle \]
of line bundles on $S$.\\
(c) (Base change) The formation of $\langle z,w \rangle$ is compatible with base change.\\
(d) (Projection formula) Let $\xx/S$ and $\yy/S$ be smooth projective families with $\xx/S$ of relative dimension $d$. Let $f \colon \xx \to \yy$ be a proper flat morphism of $S$-varieties. Let $z \in \CH^p_o(\xx/S)$ and let $w \in \CH^{d+1-p}_o(\yy/S)$. Then there exists a canonical isomorphism $\langle f_*z,w \rangle \isom \langle z, f^* w \rangle$ of line bundles on $S$.\\
\end{prop}
\begin{proof} The symmetry derives from \cite{bl}, Proposition~4, and the bi-additivity boils down to the biextension property of $\mathbb{E}$ verified in \cite{bl}, Proposition~7. The compatibility with base change is clear. The projection formula follows from \cite{sei}, Satz~7.10.
\end{proof}
Bloch's pairing is to be considered as an `intersection pairing'. In fact, let $z,w$ be classes in $\CH^p_o(\xx/S)$ and $\CH^q_o(\xx/S)$ respectively, where $p+q=d+1$. Let $Z,W$ be \emph{global} representatives of $z,w$ intersecting properly on $\xx$, and let $Z. W$ be their intersection cycle. Then there is a non-canonical isomorphism of line bundles $\langle z,w \rangle \isom \oo_S(\pi_*(Z. W))$ on $S$. In particular, if $S$ is a smooth projective curve one has $\deg_S \langle z,w \rangle =  \deg(Z . W)$.

Let $S$ again be an arbitrary complex quasiprojective variety, and consider cycle classes $z,w$ as above. Then the line bundle $\langle z,w \rangle$ carries a canonical hermitian metric $\|\cdot \|$. It is determined as follows \cite{mey}, II.4. Let $\langle Z,W \rangle$ be a local generating section of $\langle z,w \rangle$ over the non-empty Zariski open subset $U$ of $S$. In particular the cycles $Z,W$ have disjoint support over $U$. There exists a Green's current $g_W$ for the cycle $W$ on $\xx_U$ such that $ \partial \overline{\partial} g_W + \pi i \delta_W$ is a $(q,q)$-form on $\xx$, and such that $g_W$ vanishes in each fiber of $\xx/S$. Write $\langle Z,W \rangle_{\infty} = - \log \| \langle Z,W \rangle \|$. Then $\langle Z,W \rangle_{\infty}$ is given by the identity
\begin{equation} \label{archheight}
 \langle Z,W \rangle_{\infty} = - \int_\pi \delta_Z \, g_W \, ,  
\end{equation}
i.e. the `archimedean height pairing' of $Z,W$. It can be verified using Stokes' theorem that $\langle Z,W \rangle_{\infty}$ is independent of the choice of $g_W$, and is symmetric in $Z,W$. In fact we have
\begin{prop} \label{basicmetric} Each of the canonical isomorphisms of Proposition \ref{basic} is an isometry, if Bloch's pairing is endowed with the canonical hermitian metric determined by the archimedean height pairing.
\end{prop}

\section{Poincar\'e biextensions} \label{connection_biext}

An alternative approach to Bloch's line bundle valued pairing uses Poincar\'e biextensions on intermediate jacobians. References for the results in this section are \cite{hain} and \cite{ms}. Let $\pi \colon \xx \to S$ again be a smooth projective morphism of relative dimension $d$, with $S$ a complex quasiprojective variety, and consider again two positive integers $p,q$ with $p+q=d+1$. We denote by $\jj^p(\xx/S)$ and $\jj^q(\xx/S)$ the Griffiths intermediate jacobian fibrations associated to the variations of Hodge structure $\mathrm{R}^{2p-1} \pi_* \zz_{\xx}$ and $\mathrm{R}^{2q-1} \pi_* \zz_{\xx}$ on $S$. Let $\check{\jj}^p(\xx/S)$ be the dual torus fibration of $\jj^p(\xx/S)$. Then there exists a canonical isomorphism
\[ \mathrm{pd} \colon \jj^q(\xx/S) \isom \check{\jj}^p(\xx/S) \]
of torus fibrations over $S$, induced by Poincar\'e duality.

Let $\mathcal{B}$ be the canonical Poincar\'e bundle over $\jj^p(\xx/S) \times \check{\jj}^p(\xx/S)$, \cite{hain} 3.2. Here and henceforth, products are fiber products over the base variety $S$. The bundle $\mathcal{B}$ is a holomorphic line bundle on  $\jj^p(\xx/S) \times \check{\jj}^p(\xx/S)$, and comes with a trivialization $e^*\mathcal{B} \isom \oo_S$ along the zero-section $e$ of $\jj^p(\xx/S) \times \check{\jj}^p(\xx/S)$ over $S$. For classes $z,w$ in $\CH^p_o(\xx/S)$ and $\CH^q_o(\xx/S)$, respectively, we denote by $a_p(z),a_q(w)$ their Griffiths Abel-Jacobi images in $\jj^p(\xx/S)$ and $\jj^q(\xx/S)$. The pair $(a_p(z),\mathrm{pd} \, a_q(w))$ can be viewed as a `normal function' section from $S$ into $\jj^p(\xx/S) \times \check{\jj}^p(\xx/S)$. 

The connection with Bloch's pairing is expressed by Proposition \ref{bloch_biext}, which follows from \cite{ms}, Theorem~1.
\begin{prop} \label{bloch_biext} Let $z,w$ be classes in $\CH^p_o(\xx/S)$ and $\CH^q_o(\xx/S)$, respectively. Then there exists a canonical isomorphism of holomorphic line bundles
\[ \langle z, w \rangle \isom (a_p(z),\mathrm{pd} \, a_q(w))^* \mathcal{B} \]
on $S$.
\end{prop}
According to \cite{hain}, 3.2 the biextension line bundle $\mathcal{B}$ is endowed with a canonical hermitian metric. This hermitian metric is uniquely characterized by these two properties: (i) its Chern form is translation invariant in each of the fibers of $\jj^p(\xx/S) \times \check{\jj}^p(\xx/S)$ over $S$; (ii) the trivialization $e^*\mathcal{B} \isom \oo_S$ is an isometry, where $\oo_S$ is endowed with the trivial metric.

An important result is then the following.
\begin{prop} \label{bloch_biext_metric} For $\mathcal{B}$ endowed with its canonical hermitian metric, and for Bloch's pairing $\langle z,w \rangle$ endowed with the canonical hermitian metric determined by the archimedean height pairing (\ref{archheight}), the canonical isomorphism from Proposition \ref{bloch_biext} is an isometry.
\end{prop}
For a proof we refer to \cite{mey}, Section II.4.

A key example is the case where each of $p,q$ and $d$ is equal to one, i.e. where $\xx/S$ is a smooth projective morphism of relative dimension one, and $z,w$ are flat cycle classes on $\xx$ of relative dimension and degree zero. In this case $\jj^p(\xx/S)$ is the jacobian fibration $\jj(\xx/S)$ of $\xx/S$, and we have a canonical principal polarization $i \colon \jj^p(\xx/S) \isom \check{\jj}^p(\xx/S)$ as in \cite{metrperm}, D\'efinition 2.6.4. It can be derived from \cite{metrperm}, 2.7.9 that we have $\mathrm{pd}=i \circ [-1]$.

Assume there exists a relatively ample line bundle $\theta$ on $\jj(\xx/S)$ over $S$ such that $\pi_*\theta$ is a line bundle on $S$. Identify $\jj(\xx/S)$ with its dual using $i$. Then the Poincar\'e bundle on $\jj(\xx/S) \times \jj(\xx/S)$ can be explicitly written as \[ \mathcal{B} = m^*\theta  \otimes p_1^* \theta^{\otimes -1} \otimes p_2^* \theta^{\otimes -1} \otimes e^* \theta \, , \] where $m \colon \jj(\xx/S) \times \jj(\xx/S) \to \jj(\xx/S)$ is the addition map and $p_1,p_2 \colon \jj(\xx/S) \times \jj(\xx/S) \to \jj(\xx/S)$ are the projections on the first and second coordinate, respectively. The hermitian line bundle $\langle z,w \rangle$ on $S$ coincides with Deligne's pairing \cite{de} of $z$ and $w$ and we have canonical isometries
\[ \langle z,w \rangle \isom (a_1(z),-a_1(w))^*\mathcal{B} \isom (a_1(z),a_1(w))^*\mathcal{B}^{\otimes -1} \, .  \]
If $S$ is a smooth projective curve, the degree of the line bundle $(a_1(z),a_1(w))^* \mathcal{B}$ on $S$ coincides with the N\'eron-Tate height pairing between the divisor classes $z,w$ on the generic fiber of $\xx/S$. 

\section{Variations of polarized Hodge structure over $\aa_g$} \label{variation}

The constructions related to biextensions from Section \ref{connection_biext} can be carried out in a somewhat more general setting, namely general variations of (polarized) Hodge structure. We consider the case where the base space is the moduli space $\aa_g$ of complex principally polarized abelian varieties of dimension $g$. We view $\aa_g$ as an orbifold, and assume that $g \geq 2$ throughout. The basic references for this section are \cite{hain} and \cite{hrar}.

Our starting point is variations of polarized Hodge structure over $\aa_g$  given by a polarized integral Hodge structure $(V_\zz, Q : \wedge^2 V_\zz \to \zz(-n))$ of odd weight $n=-2p+1$. The corresponding variation of polarized Hodge structure over $\aa_g$ is denoted by $(\vv_\zz,\QQ)$.

We recall \cite{hrar} that $(V_\zz,Q)$ determines an intermediate jacobian fibration over $\aa_g$, which we denote by $\jj(V_\zz)$. This is a torus fibration over $\aa_g$, with fiber given as follows: let $V_{A}$ be the fiber of the local system $\vv_\zz$ at the point $[A]$ of $\aa_g$. Then the fiber of $\jj(V_\zz)$ at $[A]$ is the
complex torus
\[ J(V_A)=(V_A \otimes \CC)/(F^{-p+1} (V_A \otimes \CC) +
\mathrm{Im} \, V_A) \, . \]
This torus can be canonically identified with the extension group $\mathrm{Ext}^1(\zz,V_A)$ in the category of mixed Hodge structures.

The holomorphic tangent bundle of $J(V_A)$ comes equipped with a canonical hermitian inner product derived from~$Q$. This hermitian inner product determines a translation-invariant differential $(1,1)$-form on $J(V_A)$. It extends over $\jj(V_\zz)$ in a canonical way.
\begin{propdef} \label{unique2form}
There exists a unique $(1,1)$-form $w_V$ on the torus fibration $\jj(V_\zz)$ such that (i) its restriction to each fiber over $\aa_g$ is the translation-invariant $(1,1)$-form associated to $Q$, and (ii) its pullback along the zero-section vanishes.
\end{propdef}
\begin{proof} See \cite{hrar}, Section~5.
\end{proof}
We suppose from now on that $V_\zz$ has weight~$-1$. Let $\check{\jj}(V_\zz)$ be the complex torus fibration dual to $\jj(V_\zz)$. Let $\mathcal{B}$ be the canonical Poincar\'e biextension line bundle on $\jj(V_\zz) \times \check{\jj}(V_\zz)$, \cite{hain} 3.2. As above, this is a holomorphic line bundle on $\jj(V_\zz) \times \check{\jj}(V_\zz)$, equipped with a trivialization along the zero section, and a canonical hermitian metric.

Note that the given polarization $Q$ of $V_\zz$ induces an isogeny $i_Q \colon \jj(V_\zz) \to \check{\jj}(V_\zz)$ of torus fibrations. We define $\bb$ to be the pullback, along $(\mathrm{id},i_Q)$, of the line bundle $\mathcal{B}$ to $\jj(V_\zz)$. The bundle $\bb$ comes endowed with a trivialization along the zero section and canonical hermitian metric, both induced from $\mathcal{B}$ by pulling back. By abuse of language we refer to the hermitian line bundle $\bb$ as the biextension line bundle over $\jj(V_\zz)$.
\begin{prop} \label{biextension} Let $\vv_\zz$ be a variation of polarized
Hodge structure of weight~$-1$ over $\aa_g$ and let $\bb$ be the associated biextension line bundle over $\jj(V_\zz)$. Then the canonical hermitian metric on $\bb$ satisfies the following two properties: (i) its Chern form equals $2\,w_V$, and (ii) the pullback of $\bb$ along the zero section is, via its given trivialization, equal to the trivial hermitian line bundle. The bundle $\bb$ together with its canonical metric is uniquely characterized, as a holomorphic hermitian line bundle on $\jj(V_\zz)$, by these two properties.
\end{prop}
\begin{proof}  We refer to \cite{hrar}, Proposition~7.3 for the proof that properties (i) and (ii) hold for the canonical metric on $\bb$. The uniqueness of $\bb$ as a holomorphic hermitian line bundle on $\jj(V_\zz)$ satisfying (i) and (ii) is stated in \cite{hrar}, Proposition~6.1.
\end{proof}
We note that the constructions of Section \ref{connection_biext} fit in the present discussion, if one starts out with the universal abelian scheme $\pi \colon \xx_g \to \aa_g$ and takes the local systems $\vv_{\zz,2p-1}= \mathrm{R}^{2p-1} \pi_* \zz_{\xx_g}$ over $\aa_g$ for positive integers $p$. Then $\jj^p(\xx_g/\aa_g)=\jj(V_{\zz,2p-1})$ for each positive integer $p$ and for positive integers $p,q$ with $p+q=g+1$ one has a canonical isomorphism $\mathrm{pd} \colon \jj(V_{\zz,2q-1}) \isom \check{\jj}(V_{\zz,2p-1}) $ of torus fibrations over $\aa_g$ induced by Poincar\'e duality.

A fundamental case is the case where $(V_\zz,Q)$ equals $(H,Q_H)$, where $H=H_1(X,\zz)$ is the first homology group of a compact Riemann surface of genus $g \geq 2$, and $Q_H$ is its standard intersection form. By standard Hodge theory of complex tori, one has canonical isomorphisms $\vv_{\zz,k} \isom \wedge^k \vv_\zz$ and the polarizations on the $\vv_{\zz,k}$ are compatible with these isomorphisms.

An important role in our discussion is played by the case $k=3$. Note that the Hodge structure $H$ is mapped into $\wedge^3H$ by sending $x$ to $x \wedge \zeta $, where $\zeta$ in $\wedge^2 H$ is the dual of $Q_H$. The polarization $Q_{\wedge^3H}$ on the Hodge structure $\wedge^3 H$  sends $ (x_1\wedge x_2 \wedge x_3, y_1 \wedge y_2 \wedge y_3 )$ to $ \det(x_i,y_j)$.

We also consider the Hodge structure $\wedge^3H/H$. Its polarization is given as follows. First of all one has a contraction map $c : \wedge^3 H \to H$, defined by
\[ x \wedge y \wedge z \mapsto (x,y)z + (y,z)x + (z,x)y \, .
\]
The composite $H \to \wedge^3 H \to H$ induced by $c$ and $\wedge \zeta$ equals $(g-1)$ times the identity. Denote the projection $\wedge^3 H \to \wedge^3H/H$ by $p$.
After tensoring with $\qq$, the projection $p$ allows a splitting $j$, defined by
\[ p(x \wedge y \wedge z) \mapsto x \wedge y \wedge z - \zeta \wedge c(x\wedge y \wedge z)/(g-1) \, . \]
With these definitions, the form $Q_{\wedge^3H/H}$ on $\wedge^3H/H$ is given by the formula
\[ (u,v) \mapsto (g-1) Q_{\wedge^3 H}(j(u),j(v)) \, . \]
We denote by $w_H$, $w_{\wedge^3 H}$ and $w_{\wedge^3H/H}$ the $(1,1)$-forms on the intermediate jacobian fibrations $\jj(H)$, $\jj(\wedge^3 H)$ and $\jj(\wedge^3H/H)$ over
$\aa_g$ given by Proposition~\ref{unique2form}. The various morphisms of polarized Hodge structure described above canonically give rise to morphisms
\[ \xymatrix{ \jj(H) \ar@/^/[r]^{\wedge \zeta} & \jj(\wedge^3H) \ar[l]^c \ar[r]^p & \jj(\wedge^3H/H)
} \]
of intermediate jacobians over $\aa_g$.

The key to the proof of Theorem \ref{thmA} is the following
\begin{prop} \label{equality2biextensions}
On $\jj(\wedge^3 H)$, one has a canonical isometry
\[ \bb_{\wedge^3 H}^{\otimes g-1} \isom p^* \bb_{\wedge^3 H/H} \otimes c^* \bb_H  \]
of hermitian line bundles.
\end{prop}
\begin{proof} We have an equality $(g-1)Q_{\wedge^3 H}=p^*Q_{\wedge^3H/H}+c^*Q_H$ of polarization forms on the variation of Hodge structure associated to $\wedge^3 H$, by \cite{hrgeom}, Proposition 18. This implies the equality $(g-1)w_{\wedge^3H}=p^*w_{\wedge^3H/H}+c^*w_H$ of associated canonical $(1,1)$-forms on $\jj(\wedge^3H)$. The Chern forms of left and right hand side of the claimed isometry are therefore equal by Proposition \ref{biextension}. Moreover, both left and right hand side restrict to the trivial hermitian line bundle after pulling back along the zero section. By the uniquely defining property from Proposition \ref{biextension} we obtain the existence of the claimed isometry.
\end{proof}
Consider once again the intermediate jacobians $\jj^p(\xx_g/\aa_g)=\jj(V_{\zz,2p-1})$ for positive integers $p \leq g$ associated to the universal abelian scheme $\xx_g \to \aa_g$. Each $\jj^p(\xx_g/\aa_g)$ admits an Abel-Jacobi map $a_p \colon \CH^p_o(\xx_g/\aa_g) \to \jj^p(\xx_g/\aa_g)$. Let $p,q$ be positive integers such that $p+q=g+1$. One then has \cite{dm} a Fourier-Mukai transform $F \colon \CH^p_o(\xx_g/\aa_g) \to \CH^q_o(\xx_g/\aa_g)$ based on the Poincar\'e bundle on $\xx_g \times \xx_g$ where $\xx_g$ is identified with its dual $\check{\xx}_g$ by the tautological principal polarization. A result of Beauville implies that $F$ is compatible with the canonical polarization morphism $i \colon \jj^p(\xx_g/\aa_g) \to \check{\jj}^p(\xx_g/\aa_g)$ under the Abel-Jacobi maps.
\begin{prop} \label{beauville} Let $p,q$ be  such that $p+q=g+1$, let $F \colon \CH^p_o(\xx_g/\aa_g) \to \CH^q_o(\xx_g/\aa_g)$ be the Fourier-Mukai transform and let
$i \colon \jj^p(\xx_g/\aa_g) \to \check{\jj}^p(\xx_g/\aa_g)$ be the polarization morphism derived from the tautological polarization on $\xx_g$. Then $i(a_p(z))=\mathrm{pd} \, a_q(F(z))$ for all $z$ in $\CH^p_o(\xx_g/\aa_g)$.
\end{prop}
\begin{proof} See \cite{be}, Proposition 2.
\end{proof}

\section{Normal functions associated to the Ceresa cycle} \label{normal}

Let $\mm_g$ be the moduli space of complex smooth projective curves of genus $g \geq 2$, and let $\pi \colon \mm_{g,1} \to \mm_g$ be the universal curve. Both are viewed as orbifolds. Let $\jj(H), \jj(\wedge^3 H)$ and $\jj(\wedge^3H/H)$ be the Griffiths intermediate jacobian fibrations over $\aa_g$ determined by the local systems $H, \wedge^3 H$ and $\wedge^3H/H$, respectively, as in Section \ref{variation}. As is explained in the Introduction to \cite{hrgeom} they fit in a commutative diagram
\[ \xymatrix{ & \jj(H)   \\
\mm_{g,1} \ar[dd]_\pi \ar[ur]^\kappa \ar[r]^\mu \ar[dr]^\nu & \jj(\wedge^3 H)
\ar[u]^c \ar[d]^p
\\
 \ar[d] & \jj(\wedge^3H/H) \ar[d] \\
\mm_g \ar[r] & \aa_g }    \]
of orbifolds with $p,c$ the maps from Section \ref{variation}. The `normal functions' $\kappa, \mu$ and $\nu$ are defined as follows. First of all $\kappa$ is the map sending a pair $(X,x)$ where $X$ is a curve and $x$ is a point on $X$ to the class of the degree zero divisor $(2g-2)x - K$ in the jacobian $J$ of $X$. Here $K$ is a canonical divisor on $X$.

The map $\mu$ is B. Harris's `pointed harmonic volume' \cite{harris} sending a pair $(X,x)$ to the Abel-Jacobi image in $J(\wedge^3H)$ of the Ceresa cycle determined by $x$. Recall that the Ceresa cycle is the homologically trivial cycle in the jacobian $J$ given as $X_x-X^-_x$, where $X_x$ is the curve $X$ embedded in $J$ via the Abel-Jacobi map $y \mapsto [y-x]$ and where $X^-_x=[-1]_*X_x$. The image of $\mu(x)$ under $p$ is independent of the choice of $x$ and in fact we have \cite{pu}, Corollary 6.3 the relation
\begin{equation} \label{difference} \mu(x) - \mu(y) = -2([x-y])
\end{equation}
in $J$, sitting canonically inside $J(\wedge^3 H)$ by wedging with $\zeta$.
We find that the `harmonic volume' $\nu = p\mu$ factors over $\mm_g$, and we shall denote the resulting normal function $\mm_g \to \jj(\wedge^3H/H)$ also by $\nu$.

\section{Extension to the Picard variety of degree one over $\mm_g$} \label{extension}

Let $\Pic^d_g \to \mm_g$ be the universal Picard variety in degree~$d$, viewed as an orbifold; for example $\Pic^0_g$ is identified with the pullback of $\jj(H)$ along the Torelli map $\mm_g \to \aa_g$, i.e. the universal jacobian. The aim of this section is to show that the commutative diagram of Section \ref{normal} fits in a larger diagram
\[ \xymatrix{ & \jj(H)   \\
\Pic^1_g\ar[dd]_\pi \ar[ur]^\kappa \ar[r]^\mu \ar[dr]^\nu & \jj(\wedge^3 H)
\ar[u]^c \ar[d]^p
\\
 \ar[d] & \jj(\wedge^3H/H) \ar[d] \\
\mm_g \ar[r] & \aa_g  }   \]
with $\nu$ again factoring over $\mm_g$.
We refer to \cite{pu}, Sections 5 and 6 for the details of the following explanation. Let $X$ be a compact connected Riemann surface of genus $g \geq 2$ and as before let $H=H_1(X,\zz)$ be its first homology group. We then have the following generalization of the pointed harmonic volume of Section \ref{normal}. First of all consider the map:
\[ A \colon J(H) \longrightarrow J(\wedge^3 H) \, , \quad [x-y] \mapsto [X_y - X_x] \, . \]
This map is well-defined, and is in fact an injective homomorphism, coinciding with the embedding of $J(H)$ into $J(\wedge^3H)$ given by wedging with $\zeta$ by Lemma~6.1 of \cite{pu}. By Lemma 6.4 of \cite{pu} we also have a well-defined embedding
\[ S \colon \mathrm{Pic}^2 \, X \longrightarrow J(\wedge^3 H) \, , \quad [x+y+D] \mapsto A(D) - [X_x - X_y^-] \]
for each degree zero divisor $D$ on $X$ and $x,y \in X$. It follows that the map:
\[ \mu \colon \mathrm{Pic}^1 \, X \longrightarrow J(\wedge^3 H) \, , \quad [x+D] \mapsto -A(2D) + [X_x - X_x^-] = -A(2D)+ \mu(x) \]
is well-defined too. We have a canonical injection $X \to \mathrm{Pic}^1 X$ and the map $\mu$ clearly restricts to the pointed harmonic volume on $X$. The image $\mu(e)$ of a point $e$ in $\mathrm{Pic}^1 \, X$ represents the class of the cycle $X_e - X_e^-$ where $X_e$ is the Abel-Jacobi image of $X$ in $J(H)$ using $e$ as a base divisor.
\begin{prop} \label{linear}
(a) The map $\mu \colon \mathrm{Pic}^1\, X \to J(\wedge^3H)$ is the linear extension of the pointed harmonic volume $\mu \colon X \to J(\wedge^3H)$. \\
(b) We have the equality $\mu(e)-\mu(e') = -2(e - e')$ for all $e,e'$ in $\mathrm{Pic}^1 \, X$. In particular, the image of $\mu(e)$ under $p$ is independent of the choice of $e$. \\
(c) The set $\mu(\mathrm{Pic}^1\, X)$ is a coset of $J(H)$ in $J(\wedge^3H)$. \\
Let $K$ be a canonical divisor on $X$. \\
(d) The image of $e$ under $c\mu$ equals the class of $ (2g-2)e - K$ in $J(H)$ for all $e$ in $\mathrm{Pic}^1 \, X$.
\end{prop}
\begin{proof} \emph{(a)} Take a representative $\sum_i m_i x_i$ of $e$ with $m_i \in \zz$, $x_i \in X$ and $\sum_i m_i =1$. We would like to know whether $\mu(e)=\sum_i m_i \mu(x_i)$. For each index $i$ we can write
\[ e = [x_i + \sum_j m^{(i)}_j x_j]  \, ,  \]
where $m_j^{(i)}=m_j$ if $j \neq i$, and $m_j^{(i)}=m_i-1$ if $j=i$; note that $\sum_j m^{(i)}_j x_j$ is of degree zero. It follows that
\[ \mu(e) = -2 \, A( [\sum_j m^{(i)}_j x_j] ) + \mu(x_i) \]
for all $i$. Noting that $\mu(e) = \sum_i m_i \mu(e)$ we obtain
\[ \mu(e) = \sum_i m_i \mu(e) = -2 \, A([\sum_j (\sum_i m_i m_j^{(i)} )x_j]) + \sum_i m_i \mu(x_i) \, . \]
But $\sum_i m_i m_j^{(i)} =0$ for all $j$ and the result follows. \\
\emph{(b)} Write $e = [x+D]$ and $e'=[y+D]$ for a divisor $D$ of degree zero and for points $x,y \in X$. Equation (\ref{difference}) states that $\mu(x)-\mu(y) = -2 \, ([x-y])$ in $J(H)$. It follows that
\[ \mu(e) - \mu(e') = -A(2D) + \mu(x) + A(2D) - \mu(y) = -2([x-y]) = -2(e -e') \, , \]
as required. \\
\emph{(c)} This follows from \emph{(b)}, noting that the map $\mu \colon \mathrm{Pic}^1\, X \to J(\wedge^3H)$ has finite fibers.\\
\emph{(d)} This follows from \emph{(a)} and the fact that $c\mu(x)$ in $J(H)$ is the class of $(2g-2)x - K$ for each $x$ on $X$, as we saw in Section \ref{normal}.
\end{proof}
The proposition implies that there exist canonical normal functions $\mu$ as well as $\kappa = c\mu$ and $\nu = p\mu$ on $\Pic^1_g$, extending linearly the maps denoted by the same symbol on $\mm_{g,1}$.

\section{Proof of Theorem \ref{thmA}} \label{proofthmA}

In this section we prove Theorem \ref{thmA}. It suffices to consider the universal case, where $S$ is the moduli space $\Pic^1_g$ of pairs $(X,e)$ with $X$ a complex smooth projective curve of genus $g$ and $e$ a point on $\mathrm{Pic}^1 \, X$, and $\pi \colon \xx \to S$ is the universal curve over $\Pic^1_g$. Let $C_e=X_e-X_e^-$ be the flat family of Ceresa cycles on the universal jacobian $\jj$ over $S$. Let $F$ be the Fourier-Mukai transform on the Chow group $\mathrm{CH}^*(\jj/S)$ of relative cycles of $\jj/S$. We note that both $C_e$ and $F(C_e)$ are homologically trivial in the fibers of $\jj/S$.

Proposition \ref{first} establishes a connection between $C_e$ and the flat family of Gross-Schoen cycles $\Delta_e$ on the universal triple product curve $\yy$ over $S$. The proposition generalizes a result on the level of $\rr$-valued height pairings from \cite{zh}, Section~5.
\begin{prop} \label{first}
There exists a canonical isometry
\[   \langle \Delta_e, \Delta_e \rangle^{\otimes 2} \isom \langle C_e, F(C_e) \rangle^{\otimes 3}   \]
of hermitian line bundles on $\Pic^1_g$.
\end{prop}
\begin{proof}  We have a proper flat map $f \colon \yy \to \jj$ over $S$ given by sending a tuple $(X,e;x,y,z)$ to the class of $x+y+z-3e$ in the jacobian of $X$. We have $f_*\Delta_e \equiv 3C_e$ in $\mathrm{CH}^{g-1}_o(\jj/S)$ by \cite{cvg}, Proposition 2.9. Further we have $f^*F(C_e) \equiv 2\Delta_e$ in $\mathrm{CH}^2_o(\yy/S)$ by \cite{zh}, Theorem 1.5.5. Propositions \ref{basic} and \ref{basicmetric} then give a chain of canonical isometries
\[ \langle \Delta_e,\Delta_e \rangle^{\otimes 2} \isom \langle \Delta_e, f^* F(C_e) \rangle \isom \langle f_* \Delta_e, F(C_e) \rangle \isom
\langle C_e, F(C_e) \rangle^{\otimes 3}  \]
of hermitian line bundles on $\Pic^1_g$. The middle isometry is an application of the projection formula.
\end{proof}
Let $\kappa$, $\mu$ and $\nu$ be the normal functions on $\Pic^1_g$ constructed as a result of the discussion in Section \ref{extension}.
\begin{prop} \label{second} There exists a canonical isometry
\[ \langle C_e, F(C_e) \rangle \isom \mu^* \bb_{\wedge^3 H}  \]
of hermitian line bundles on $\Pic^1_g$.
\end{prop}
\begin{proof} By Proposition \ref{bloch_biext_metric} we have a canonical isometry
\[ \langle C_e,F(C_e) \rangle \isom (a(C_e),\mathrm{pd} \, a(F(C_e)))^* \mathcal{B}_{\wedge^3H} \]
where $\mathcal{B}_{\wedge^3H}$ is the Poincar\'e biextension line bundle on $\jj(\wedge^3 H ) \times \check{\jj}(\wedge^3 H)$. Here we use $a$ as a shorthand for Abel-Jacobi image. By Proposition \ref{beauville} and by the construction of $\bb_{\wedge^3H}$ we find a chain of equalities
\[ (a(C_e),\mathrm{pd} \, a(F(C_e)))^* \mathcal{B}_{\wedge^3H} = (a(C_e),ia(C_e))^* \mathcal{B}_{\wedge^3 H} = \mu^* \bb_{\wedge^3 H} \, . \]
We obtain the proposition by combining these two results.
\end{proof}
\begin{prop} \label{third}
There exists a canonical isometry
\[ \mu^* \bb_{\wedge^3 H}^{\otimes g-1} \isom \nu^* \bb_{\wedge^3 H/H} \otimes \kappa^*\bb_H \]
of hermitian line bundles on $\Pic^1_g$.
\end{prop}
\begin{proof} By Proposition \ref{equality2biextensions} we have a canonical isometry
\[ \bb_{\wedge^3 H}^{\otimes g-1} \isom p^* \bb_{\wedge^3H/H} \otimes c^*\bb_H  \]
of line bundles over $\jj(\wedge^3 H)$. We obtain the required isometry over $\Pic^1_g$
by pulling back this isometry along the extended harmonic volume $\mu$, noting that $\kappa=c\mu$ and $\nu=p\mu$.
\end{proof}
\begin{prop} Let $x_e$ be the relative degree zero divisor class on the universal curve $\xx$ over $\Pic^1_g$ given by $(2g-2)e-c_1(\omega)$, where $\omega$ is the relative dualizing sheaf of $\xx/\Pic^1_g$ and $e$ is the tautological relative degree one divisor. Then there exists a canonical isometry
\[ \kappa^* \bb_H \isom \langle x_e, x_e \rangle^{\otimes -1} \]
of hermitian line bundles on $\Pic^1_g$.
\end{prop}
\begin{proof} By construction $\bb_H$ is the pullback along $(\mathrm{id},i)$ of the Poincar\'e bundle $\mathcal{B}_H$ on $\jj(H) \times \check{\jj}(H)$. 
We have seen at the end of Section \ref{connection_biext} that $\mathrm{pd}=i \circ [-1]$ in this case. Note that the Abel-Jacobi image of $x_e$ over $(X,e)$ is equal to $\kappa(X,e)$. We obtain a chain of isometries
\[ \kappa^*\bb_H \isom (a(x_e),ia(x_e))^* \mathcal{B}_H \isom
(a(x_e),\mathrm{pd} \, a(x_e))^* \mathcal{B}_H^{\otimes -1} \isom \langle x_e,x_e \rangle^{\otimes -1} \, , \]
with the last isometry given by Proposition \ref{bloch_biext_metric}. The proposition follows.
\end{proof}
\begin{proof}[Proof of Theorem \ref{thmA}] By combining the above propositions we obtain a canonical isometry of hermitian line bundles
\[  \langle \Delta_e,\Delta_e \rangle^{\otimes 2g-2} \isom \nu^*\bb^{\otimes 3}_{\wedge^3 H/H} \otimes \langle x_e,x_e \rangle^{\otimes -3} \]
on $\Pic^1_g$. The proof of Theorem \ref{thmA} is finished once we show that
there exists an up to sign canonical isomorphism of line bundles
\[  \nu^*\bb_{\wedge^3 H/H}^{\otimes 3} \isom \langle \omega, \omega \rangle^{\otimes 2g+1} \]
on $\mm_g$ of norm $\exp(-(2g-2)\varphi)$, where $\varphi$ is the invariant defined in (\ref{defphi}). Here $\omega$ is the relative dualizing sheaf of the universal curve over $\mm_g$, and $\langle \omega,\omega \rangle$ has the metric induced by the Arakelov metric on $\omega$ \cite{de}. By a result of Morita \cite{hrgeom}, Theorem 7 there exists an isomorphism of holomorphic line bundles
\[ \nu^*\bb_{\wedge^3 H/H} \isom (\det \pi_* \omega)^{\otimes 8g+4}  \]
on $\mm_g$. As the only invertible holomorphic functions on $\mm_g$ are constants, \cite{hrar} Lemma 2.1, the isomorphism is determined up to a scalar.

We can determine this scalar modulo a sign in a canonical way as follows. Consider the hyperelliptic locus $\hh_g$ inside $\mm_g$. Then the line bundle $(\det \pi_* \omega)^{\otimes 8g+4}$ restricted to $\hh_g$ has an up to sign canonical global section $\Lambda_g$, characterized by its unique extension as a global nowhere vanishing section of $(\det \pi_* \omega)^{\otimes 8g+4}$ over the moduli stack of hyperelliptic curves of genus $g$ over the integers.

On the other hand, the map $\nu$ restricted to $\hh_g$ coincides with the zero section, \cite{hrar} Proposition 6.7. It follows that the line bundle $\nu^*\bb_{\wedge^3 H/H}$ restricted to $\hh_g$ has a canonical trivialization. The scalar in the Morita isomorphism is therefore determined, up to sign, by letting the trivializing sections of $\nu^*\bb_{\wedge^3 H/H}$ and of $(\det \pi_* \omega)^{\otimes 8g+4}$ over $\hh_g$ correspond.

We next compute the norm of the canonical Morita isomorphism, where $\det \pi_* \omega$ has the metric induced from the $L^2$-metric on $\pi_* \omega$ given by $\|\alpha\|^2 = \frac{i}{2}\int \alpha \wedge \overline{\alpha}$. Let $\delta_F$ be Faltings's delta-invariant \cite{fa} on $\mm_g$ and put $\delta=-4g\log(2\pi)+\delta_F$. Let $\lambda$ be Zhang's lambda-invariant \cite{zh}, Section~1.4, determined by the equality
\[ \lambda = \frac{g-1}{6(2g+1)}\varphi +\frac{1}{12}\delta \, . \]
Then the main result of \cite{dj} implies that the norm of the canonical Morita isomorphism is equal to $\exp(-(8g+4)\lambda)$.

Note that there is an up to sign canonical isomorphism $(\det \pi_* \omega)^{\otimes 12} \isom \langle \omega, \omega \rangle$ of line bundles on $\mm_g$ due to Mumford. According to Faltings \cite{fa}, Theorem 6 and Moret-Bailly \cite{mb}, Th\'eor\`eme 2.2 the norm of the Mumford isomorphism equals $\exp(\delta)$, if $\langle \omega, \omega \rangle$ is endowed with the metric induced from the Arakelov metric on $\omega$. Combining the Mumford and Morita isomorphisms we obtain an up to sign canonical isomorphism $ \nu^*\bb_{\wedge^3 H/H}^{\otimes 3} \isom \langle \omega, \omega \rangle^{\otimes 2g+1}$ which has norm $\exp(-(2g-2)\varphi)$, as required.
\end{proof}
It would be interesting to have a way of determining the Morita isomorphism (and hence the isomorphism in Theorem \ref{thmA}) that does not refer to the locus of hyperelliptic curves.

\section{Extension to stable curves of compact type} \label{furtherextension}

Let $\mm_g^c$ denote the orbifold of stable curves of genus $g\geq 2$ of compact type, i.e. of stable curves such that the associated jacobian has trivial toric part. Let $\Pic^{1,c}_g \to \mm_g^c$ denote the universal degree one part of the Picard variety over $\mm_g^c$. The objective of this section is to extend the hermitian line bundle $\langle \Delta_e,\Delta_e \rangle$ over $\Pic^{1,c}_g$, and to prove Theorems \ref{thmB} and \ref{thmC}. A basic reference for this section is \cite{hain_normal}, esp. Section~5.

First of all note that by pulling back the various variations of polarized Hodge structure on $\aa_g$ constructed in Section~\ref{variation} along the Torelli map $\mm_g^c \to \aa_g$ we obtain variations of polarized Hodge structure over $\mm_g^c$, together with their associated intermediate jacobian fibrations $\jj(V_\zz)$ endowed with a canonical metrized biextension line bundle and canonical invariant $(1,1)$-form $w_V$. By the discussion in Section~5 of \cite{hain_normal}, and linear extension, each of the normal functions $\kappa$, $\mu$, $\nu$ defined above on $\Pic^1_g$ extends as a normal function into the appropriate intermediate jacobian fibration over $\mm^{c}_g$. In particular we have a canonical hermitian line bundle $\mu^*(\bb_{\wedge^3 H})$ over $\Pic^{1,c}_g$. It is the unique extension of the hermitian line bundle $\mu^*(\bb_{\wedge^3 H})$ over $\Pic^1_g$ as a hermitian line bundle over $\Pic^{1,c}_g$.  

By Propositions \ref{first} and \ref{second} above we have a canonical isometry $\langle \Delta_e,\Delta_e \rangle^{\otimes 2} \isom \mu^*(\bb^{\otimes 3}_{\wedge^3 H})$ over $\Pic^1_g$. We conclude that there exists a unique extension of the Bloch pairing $\langle \Delta_e,\Delta_e \rangle$ as a hermitian line bundle over $\Pic^{1,c}_g$. In particular we have a canonical extension of $\langle \Delta_e,\Delta_e \rangle$ as a line bundle over $\Pic^{1,c}_g$. Almost by construction, the following proposition holds.
\begin{prop} \label{GSaspullback}
We have a canonical isometry $\langle \Delta_e,\Delta_e \rangle^{\otimes 2} \isom \mu^*(\bb^{\otimes 3}_{\wedge^3 H})$ of hermitian line bundles over $\Pic^{1,c}_g$.
\end{prop}
Theorem \ref{thmB} is obtained by combining Proposition \ref{GSaspullback} with the following general result.  
\begin{thm} \label{hain}
Let $M$ be a complex manifold and let $\vv$ be a variation of polarized integral Hodge structures of weight $-1$ on $M$. Let $\jj(V) \to M$ be the associated intermediate jacobian fibration, let $\mu \colon M \to \jj(V)$ be a normal function and let $w_V$ be the canonical invariant $(1,1)$-form on $\jj(V)$. Then $\mu^*w_V$ is a non-negative $(1,1)$-form on $M$.
\end{thm}
\begin{proof} See \cite{hain_normal}, Theorem 13.1.
\end{proof}
\begin{proof}[Proof of Theorem \ref{thmB}]
From Proposition \ref{GSaspullback} it follows that the Chern form of $\langle \Delta_e,\Delta_e \rangle^{\otimes 2}$ and the Chern form of $\mu^*(\bb_{\wedge^3 H}^{\otimes 3})$ are equal on $\Pic^{1,c}_g$. By Proposition \ref{biextension}, the Chern form of $\mu^*(\bb_{\wedge^3 H})$ is equal to $2\, \mu^*w_{\wedge^3 H}$ where $w_{\wedge^3 H}$ is the canonical invariant $(1,1)$-form on $\jj(\wedge^3 H)$. Theorem \ref{hain} then implies that the Chern form of $\mu^*(\bb_{\wedge^3 H})$ is non-negative.
\end{proof}
Our proof of Theorem \ref{thmC} is based on class calculations by Hain-Reed \cite{hrar} and Hain \cite{hain_normal}.
\begin{proof}[Proof of Theorem \ref{thmC}] Our starting point is the canonical isomorphism
\[ \bb_{\wedge^3 H}^{\otimes g-1} \isom p^* \bb_{\wedge^3H/H} \otimes c^*\bb_H  \]
of line bundles on the intermediate jacobian fibration $\jj(\wedge^3 H)$ over $\aa_g$, provided by Proposition \ref{equality2biextensions}. By pulling back along the map $\mu \colon \mm^{c}_{g,1} \to \jj(\wedge^3 H)$ we obtain, analogously to Proposition \ref{third}, a canonical isomorphism
\[ \mu^* \bb_{\wedge^3 H}^{\otimes g-1} \isom \nu^* \bb_{\wedge^3 H/H} \otimes \kappa^*\bb_H \]
of line bundles on $\mm^c_{g,1}$. By Theorem 1.3 of \cite{hrar} we have the equality
\[ \nu^* \bb_{\wedge^3 H/H} = (8g+4)\lambda -  \sum_{i=1}^{g-1} 4i(g-i) \delta_i^{ \{x\} } \]
in $\mathrm{Pic}(\mm^c_{g,1})$ and by Theorem 10.2 of \cite{hain_normal} we have
\[ \kappa^* \bb_H = 4g(g-1)\psi - 12 \lambda - \sum_{i=1}^{g-1} 4i(i-1) \delta_{g-i}^{ \{x\} } \, . \]
By combining we obtain
\[ \mu^* \bb_{\wedge^3 H} = 4g \psi + 8 \lambda - \sum_{i=1}^{g-1} 4i \delta_{g-i}^{ \{x\} } \, . \]
Using Proposition \ref{GSaspullback} we obtain the required expression for $\langle \Delta_e,\Delta_e \rangle$.
\end{proof}

\vspace{1cm}

\noindent Address of the author: \\  \\
Robin de Jong \\
Mathematical Institute \\
University of Leiden \\
PO Box 9512 \\
2300 RA Leiden \\
The Netherlands \\
Email: \verb+rdejong@math.leidenuniv.nl+


\begin{thebibliography}{99}

\bibitem{ar} S. Y. Arakelov, \emph{An intersection theory for divisors on an
arithmetic surface}. Izv. Akad. USSR 86 (1974), 1164--1180.

\bibitem{ac} E. Arbarello and M. Cornalba, \emph{The Picard groups of the moduli spaces of curves}. Topology 26 (1987), 153--171.

\bibitem{be} A. Beauville, \emph{Quelques remarques sur la transformation de Fourier dans l'anneau de Chow d'une vari\'et\'e abelienne}. In: Algebraic Geometry--Tokyo/Kyoto 1982, Lecture Notes in Mathematics 1016 (1983), 238--260.

\bibitem{beil} A. Beilinson, \emph{Height pairing between algebraic cycles}. In: Current trends in arithmetical algebraic geometry (Arcata, Calif., 1985), Contemp. Math. 67 (1987), 1--24.

\bibitem{bljpaa} S. Bloch, \emph{Height pairing for algebraic cycles}. J. Pure Appl. Algebra 34 (1984), 119--145.

\bibitem{bl} S. Bloch, \emph{Cycles and biextensions}. In: Algebraic $K$-theory and algebraic number theory (Honolulu, HI, 1987), Contemp. Math. 83 (1989), 19--30. 

\bibitem{cvg} E. Colombo, B. van Geemen, \emph{Note on curves in a jacobian}.
Compositio Math. 88 (1993), 333--353.

\bibitem{de} P. Deligne, \emph{Le d\'eterminant de la cohomologie}. In: Current trends in arithmetical algebraic geometry (Arcata, Calif., 1985), Contemp. Math. 67 (1987), 93--177.

\bibitem{dm} C. Deninger, J. Murre, \emph{Motivic decomposition of abelian schemes and the Fourier transform}. J. Reine Angew. Math. 422 (1991), 201--219.

\bibitem{fa} G. Faltings, \emph{Calculus on arithmetic surfaces}. Ann. of Math. 119 (1984), 387--424.

\bibitem{gs} B. Gross, C. Schoen, \emph{The modified diagonal cycle on the triple product of a pointed curve}. Ann. Inst. Fourier 45 (1995), 649--679.

\bibitem{hain} R. Hain, \emph{Biextensions and heights associated to curves of odd genus}. Duke Math. J. 61 (1990), 859--898.

\bibitem{hrgeom} R. Hain, D. Reed, \emph{Geometric proofs of some results of
Morita}. J. Algebraic Geometry 10 (2001), 199--217.

\bibitem{hrar} R. Hain, D. Reed, \emph{On the Arakelov geometry of the moduli
space of curves}. J. Differential Geom.  67 (2004), 195--228.

\bibitem{hain_normal} R. Hain, \emph{Normal functions and the geometry of moduli spaces of curves}. To appear in G. Farkas and I. Morrison (eds.), Handbook of Moduli. Preprint, \verb+arxiv:1102.4031+.

\bibitem{harris} B. Harris, \emph{Harmonic volumes}. Acta Math. 150 (1983), 91--123.

\bibitem{dj} R. de Jong, \emph{Second variation of Zhang's $\lambda$-invariant on the moduli space of curves}. To appear in Amer. Jnl. Math. Preprint, \verb+arxiv:1002.1618+.

\bibitem{mey} O. Meyer, \emph{\"Uber Biextensionen und H\"ohenpaarungen algebraischer Zykel}. PhD Thesis, Universit\"at Regensburg, 2003.

\bibitem{mb} L. Moret-Bailly, \emph{La formule de Noether pour les surfaces arithm\'etiques}. Invent. Math. 98 (1989), 491--498. 

\bibitem{metrperm} L. Moret-Bailly, \emph{M\'etriques permises}. In: S\'eminaire sur les pinceaux arithm\'etiques, Ast\'erisque 127 (1985), 29--88.

\bibitem{ms} S. M\"uller-Stach, \emph{$\CC^*$-extensions of tori, higher Chow groups and applications to incidence equivalence relations for algebraic cycles}. $K$-Theory 9 (1995), 395--406. 

\bibitem{pu} M. J. Pulte, \emph{The fundamental group of a Riemann surface: mixed Hodge structures and algebraic cycles}. Duke Math. J. 57 (1988), no. 3, 721--760.

\bibitem{sei} M. Seibold, \emph{Bierweiterungen f\"ur algebraische Zykel und Poincar\'ebundel}. PhD Thesis, Universit\"at Regensburg, 2007.


\bibitem{zhadm} S. Zhang, \emph{Admissible pairing on a curve}. Invent. Math. 112 (1993), 171--193.

\bibitem{zh} S. Zhang, \emph{Gross-Schoen cycles and dualising sheaves}.
Invent. Math. 179 (2010), 1--73.

\bibitem{zh2} S. Zhang, \emph{Positivity of heights of codimension 2 cycles over function field of characteristic~0}. Preprint, \verb+arXiv:1001.4788+.

\end{thebibliography}
\end{document}